\documentclass[reqno]{amsart}
\usepackage{amssymb, amsmath, amsthm}
\usepackage[letterpaper, margin=1in]{geometry}
\usepackage{enumitem}
\usepackage{esint, nicefrac}

\newtheorem{theorem}{Theorem}[section]
\newtheorem{lemma}[theorem]{Lemma}
\newtheorem{proposition}[theorem]{Proposition}

\theoremstyle{definition}
\newtheorem{definition}{definition}[section]
\theoremstyle{remark}
\newtheorem{remark}[definition]{Remark}

\numberwithin{equation}{section}

\newcommand{\norm}[1]{\left\lVert#1\right\rVert}
\newcommand{\R}{\mathbb{R}}

\DeclareMathOperator{\dist}{dist}
\newcommand\res{\mathop{\hbox{\vrule height 7pt width .3pt depth 0pt\vrule height .3pt width 5pt depth 0pt}}\nolimits}
\usepackage[colorlinks=true,linkcolor=black]{hyperref}
\usepackage{bookmark}
\title[Second derivatives a.e. for rank-one convex functions]{A note on the a.e. second-order differentiability of rank-one convex functions}
\author{Jonas Hirsch}

\begin{document}
\begin{abstract}
    In the Euclidean setting, the well-known Alexandrov theorem states that convex functions are twice differentiable almost everywhere. In this note, we extend this theorem to rank-one convex functions. Our approach is novel in that it draws more from viscosity techniques developed in the context of fully nonlinear elliptic equations. 
    As a byproduct, the original Alexandrov theorem can essentially be reduced to the a.e. differentiability of one-dimensional monotone functions, as presented in the appendix.
\end{abstract}
\maketitle
\section{Introduction}
The aim of this very short note is to show that rank-one convex functions are a.e. second order differentiable:

\begin{proposition}\label{prop.alexandrov}
Let $f\colon B_1 \subset \R^{m\times n} \to \R$ be a bounded, rank-one convex function, then $f$ is second order differentiable for a.e. $x \in B_{\nicefrac{1}{2}}$.
\end{proposition}

Let us shortly recall that a function $f\colon B_1 \to \R$ is called rank-one convex if its restrictions to line segments contained in $B_1$ in rank-one  directions are convex i.e. 
\[ t\mapsto f(x+t(a\otimes b)) \text{ is convex } \forall x\in B_1, a\in \R^m,b\in \R^n \text{ on the interval of all $t$ with } |t||a||b|<\dist(x,\partial B_1)\,. \]
\begin{remark}\label{rem.separate convex}

With a few minor modifications to Lemma \ref{lem.upper lower bound}, the same conclusion can be obtained for separate convex functions. This result may seem surprising at first due to the work of Conti, Faraco, Maggi and M\"uller \cite{ContiFaracoMaggiMueller2005} and Kirchheim and Kristensen \cite{KirchheimKristensen2016}. This shows that almost everywhere second-order differentiability is not related to whether the Hessian is a measure. 
\end{remark}

\section{proof}
 
\subsubsection*{Ingredient 1: $f$ has a paraboloid tangent from above a.e. in $B_1$}\,

\noindent Since for any $x \in \R^{m\times n}$ and any $i = 1,\dotsc, m, j=1,\dotsc,n$ one has $x_{ij}=(e_i\otimes e_j)\colon x $, the rank-one convexity implies that in the weak and viscosity sense \[\partial^2_{x_{ij}}f(x)= \frac{d^2}{dt^2}|_{t=0} f(x+t(e_i\otimes e_j))\ge 0\footnote{ in case of separate convexity one only uses $(e_i \otimes e_i), i=1,\dotsc, n$}\,.\]
Hence we deduce that $f$ is sub-harmonic in the viscosity sense i.e. 
\[ \Delta f(x) \ge 0\,.\] 

\begin{remark}
     The proposition also holds true in the case of $B_1 \subset \R^{n\times n}_{\operatorname{sym}}$. In this case, the rank-one directions are taken as $r_{ij}= (e_i+e_j) \otimes (e_i+e_j)$ for $i,j =\{1,\dots,n\}$. These directions provide a basis since $2\operatorname{sym}(e_i\otimes e_j)=r_{ij}-r_{ii}-r_{jj}$. We can now replace the Laplacian with the following two-homogeneous elliptic coefficient operator: 
    \[L f (x)= \sum_{i,j} \partial^2_{r_{ij}}f(x)= \sum_{(kl),(mn)} a^{(kl)(mn)} \partial_{x_{kl}\,x_{mn}}f(x),\]
    where $a = \sum_{i,j} r_{ij} \otimes r_{ij}$.
    
\end{remark}

The statements concerning the viscosity can be verified by mollification since rank-one convexity is preserved under mollification and viscosity solutions are closed under uniform convergence. 

Thus, we can apply Lin's ``classical'' interior $W^{2,\epsilon}$ estimate for viscosity subsolutions to fully nonlinear uniformly elliptic equations, e.g. \cite[Theorem 2.1]{Mooney2019} originally \cite{Lin1986},
\begin{equation}\label{e.interior W2epsilon}
    \left|\left\{\overline{\Theta}_f> C \norm{f}_{L^\infty(B_1)}t\right\}\cap B_{\nicefrac{1}{2}}\right|\lesssim t^{-\epsilon}\,.
\end{equation}

To clarify the used notation in the above estimate: One says that a paraboloid $P$ has opening $A$ if $D^2P= A\,I$. For a continuous function $u \in C(\overline{\Omega})$ one sets $\overline{\Theta}_u(x)$ to be the smallest  $a\ge 0$ such that a paraboloid of opening $a$ is tangent from above to $u$ in $\Omega$ at $x$ i.e. there is $P$ with $D^2P = a I$ such that $P \ge u$ in $\overline{\Omega}$ and $u(x)=P(x)$. In the inequality above we have taken the domain to be $\overline{B}_1$.




\subsubsection*{Ingredient 2: upper bound implies lower bound}\,

\noindent
In this section we want to show the following Lemma:
\begin{lemma}\label{lem.upper lower bound}
There is a dimensional constant $C=C(n)$ such that if a rank-one convex function $f$ satisfies $f(x_0)=0, Df(x_0)=0$ and there is a non-decreasing $g \in C^0(\R)$ such that $G(x)=g(|x-x_0|)$ is tangent from above to $f$ in $\Omega$ at $x_0$ then 
\begin{equation}\label{e.lower bound}
    f(x)\ge - C \,G(x) \quad \forall x \in \Omega\,.
\end{equation}
\end{lemma}
\begin{proof} Passing to $x\mapsto f(x_0+x)$ and $\Omega_{x_0}=\Omega-x_0$ we may assume that $x_0=0$.

Let $x\in \R^{m\times n}$ be given, let $x_i$ be the matrix of the first $i$th columns of $x$ i.e. $x_i = (x_{\cdot1},\dotsc, x_{\cdot i}, 0 \dotsc0)$. Furthermore we define the rank-one matrices build by the $i$th column $d_i= x_{\cdot i }\otimes e_i $. Hence we have $x_i=\frac12 x_{i+1} +\frac12 y_{i+1}$, where $y_i=x_i-2d_i$. Thus the rank-one convexity of $f$ implies that 
\[ 2 f(x_i) \le f(x_{i+1}) + f(y_i) \le f(x_{i+1}) + G(x)\,,\]
using for the second inequality that $f\le G$ and the monotonicity of $g$ together with $|y_i|=|x_i|\le |x|$ for all $i$. 

Finally, one has $f(x_1) \ge f(0)+ Df(0) x_1 \ge 0$ since $x_1$ is a rank-one matrix. Thus we can use the above inequality to deduce \eqref{e.lower bound} inductively, since $x_n =x$.
\end{proof}

\begin{proof}[Proof of the Proposition]
Recall that every rank-one convex function is locally Lipschitz, with the quantitative estimate
\begin{equation}\label{e.Lipschitzbound}
    \operatorname{Lip}(f,B_r(x)) \le n \frac{\operatorname{osc}(f, B_{2r}(x))}{r}\,.
\end{equation}
Due to Rademacher's theorem $f$ is therefore a.e. differentiable i.e. there $N \subset B_2$ with $|N|=0$ and $f$ is differentiable on $B_2\setminus N$ with $|Df(x)| \le C \norm{f}_{L^\infty(B_1)}$ for all $x\in (B_{\nicefrac{3}{2}}\setminus N)$. 
\medskip

\emph{Step 1: $f_{x_{ij}}=\partial_{x_{ij}}f$ can be touched from above and below by a cone with an opening of $CA$ on the set $\Omega_A=\{\overline{\Theta}_f\le A\}\cap (B_{\nicefrac{1}{2}}\setminus N)$}
\smallskip

Given $x_0 \in \Omega_A$ we consider the rank-one convex function $\tilde{f}=f-f(x_0)-Df(x_0)(x-x_0)$. By construction $\tilde{f}(x_0)=0, D\tilde{f}(x_0)=0$ 
and since $x_0 \in \Omega_A$ we have $\tilde{f}(x) \le \frac{A}{2}|x-x_0|^2$ on $B_1$. In particular, this implies that the assumptions of Lemma \ref{lem.upper lower bound} are satisfied with $g(t)=\frac{A}{2}t^2$. 
Hence $\lVert\tilde{f}\rVert_{L^\infty(B_r(x_0))}\le CA r^2$ for all $r<\nicefrac{1}{2}$. 
Combining it with \eqref{e.Lipschitzbound} we deduce for any $x\in B_1$ with $r=|x-x_0|<\nicefrac{1}{4}$ that
\[|Df(x)-Df(x_0)|= |D\tilde{f}(x)|\le \operatorname{Lip}(\tilde{f},\overline{B}_{r}(x_0))\le C \frac{\lVert\tilde{f}\rVert_{L^\infty(B_{2r}(x_0))}}{r} \le CA\,r\,.\]

\emph{Step 2: $f_{x_{ij}}$ is differentiable a.e. in $\Omega_A$} 
\smallskip

The following argument is our version of Maly's beautiful argument \cite{Maly1999}, which unfortunately we only obtained afterwards. However, as his article is difficult to obtain, we have decided to present our version nonetheless. Our version uses a sub- and super-convolution that is closer to the one used by Jensen \cite{Jensen}.

We consider the sup- and inf-convolution with cones i.e. we define on $B_{\nicefrac{1}{2}}$ for $L=2C \max\{ A,\norm{f}_{L^\infty(B_1)}\} $ 
\begin{align*}
    w_{x_{ij}}^-(x)=\inf\{ f_{x_{ij}}(y)+L |x-y| \colon y \in B_{\nicefrac{3}{4}}\setminus N\}\\
    w_{x_{ij}}^+(x)=\sup\{ f_{x_{ij}}(y)- L |x-y| \colon y \in B_{\nicefrac{3}{4}}\setminus N\}\,.
\end{align*}
Firstly, we note that $\inf, \sup$ are essentially taken over $|x-y|<\frac14$, since $|y-x|\ge \frac14$ by our choice of $L$ and the Lipschitz estimate, \eqref{e.Lipschitzbound}, we have for every $x \in B_{\nicefrac{1}{2}}\setminus N$ that 
\[ f_{x_{ij}}(x) \le C \norm{f}_{L^\infty(B_1)} < f_{x_{ij}}(y)+L |x-y|\,.\]
Hence the infimum in $w_{x_{ij}}^-$ can be taken only over $|y-x|<\frac14$ as claimed. 
Secondly, its not difficult to see that $w_{x_{ij}}^- \le f_{x_{ij}} \le w_{x_{ij}}^+$ and that $w_{x_{ij}}^\pm$ are $L$-Lipschitz continuous. Furthermore, Step 1 implies that $w_{x_{ij}}^-=f_{x_{ij}}=w_{x_{ij}}^+$ on $\Omega_A$. Hence $w_i^\pm$ for a.e. are differntiable $x_0\in \Omega_A$ but since they are ordered their differential must agree, i.e. $Dw_i^\pm(x_0)=M$ for some $M\in \R^{m\times n}$. This implies the differentiablity of $f_{x_{ij}}$ in $x_0$, again because they are ordered $w_{x_{ij}}^- \le f_{x_{ij}} \le w_{x_{ij}}^+$. 

\emph{Step 3: $f$ is second-order differentiable at a.e. point of $\Omega_A$.}
\smallskip

Let $x_0\in \Omega_A$ be a point in which all $f_{x_{ij}}$ are differentiable, then for any $z$ with $|z|\le \nicefrac{1}{4}$ we have by the fundamental theorem of calculus, which is valid for Lipschitz functions: using Einstein summation we have
\begin{align*}
    &f(x_0+z) - (f(x_0)+Df(x_0)z+\frac12Df_{x_{ij}}(x_0)z_{ij}z)\\&=\int_0^1 \left(f_{x_{ij}}(x_0+sz)-f_{x_{ij}}(x_0)-Df_{x_{ij}}(x_0)\right)z_{x_{ij}}(sz)\,ds\\& = o(|z|^2)\,.
\end{align*}

\end{proof}

\appendix
\section{A measure theoretic argument for the convex case}
Alexandrov's original theorem can essentially be reduced to two observations. Firstly, the one-dimensional situation can be reduced to the fact that monotone functions are differentiable almost everywhere (a result already discovered by Lebesgue). Secondly, the convex hull of the set $\{ \pm h e_j \colon j=1,\dotsc, n\}$ contains the open ball $B_{\nicefrac{h}{\sqrt{n}}}$. 

More precisely, we use a quantitative version of the almost everywhere differentiability of monotone functions — which was probably also known to Lebesgue — and the second observation to derive \eqref{e.interior W2epsilon} with $\epsilon=1$. In the following, $Q_r(x)$ denotes the cube $x+Q_r$, where $Q_r=[-r,r]^n$.

\emph{Measure theoretic ingredient: } Given a Borel measure $\mu$ on $\R$ we consider the associated maximal function i.e. 
\[ M\mu(x)= \sup\left\{ \frac{\mu(I)}{|I|} \colon x \in I =(a,b)\right\}\,.\]
The classical maximal-function estimate/ Lebesgue differentiation estimate states 
\[ |\{ M\mu > t \}|\lesssim \frac{\mu(\R)}{t}.\]
Applying the above to the truncated measure $\tilde{\mu}=\mu\res[-2,2]$ gives a localized version: for every $t>\mu[-2,2]$ and any interval $I$ with $\mu(I)> t|I|$ we have $|I|<1$ hence we deduce that 
\[ |\{M(\mu)>t\}\cap[-1,1]|\le |\{M\tilde{\mu}>t\}| \lesssim \frac{\mu[-2,2]}{t}\]

\emph{Application to the Hessian of a convex function: } 
Let $f''$ be the Radon measure provided by the second derivative of a convex function $f$ on $\R$. then we have 
\[ f''[-2,2]=f'(2)-f'(-2) \le 2 \operatorname{osc}(f, [-3,3])\,.\]
Furthermore, if $f(0)$ with $0 \in \partial f(0)$ for instance by approximation with mollification, one extends the classical Taylor approximation to 
\begin{equation}\label{e.taylor for convex}
    0\le f(h)\le f''[0,h] h \le Mf''(0) \,h^2 \text{ and  }0\le f(-h)\le f''[-h,0] h \le Mf''(0)\,h^2 \quad \forall h >0\,.
\end{equation}

\emph{Derivation of \eqref{e.interior W2epsilon}: } For a fixed direction $e_i$ and $y \in e_i^\perp$ we may consider the convex function $s\mapsto f_y(s)=f(y+s e_i)$ and the associated set $E_y=\{ Mf_y''>t\}\cap [-1,1]$. From the above, we have 
\[|E_y|\lesssim t^{-1} \operatorname{osc}(f_y,[-3,3])\,.\]
Hence, for each $t>2\operatorname{osc}(f,Q_3)$ we can combine them to $E_t=N\cup\bigcup_{i=1}^n E_i \subset Q_1$, where $E_i=\bigcup_{y \in e_i^\perp \cap Q_1} E_y$ and $N$ is the set of measure zero where $f$ is not differentiable. Using Fubini, we can estimate its measure from above by
\[ |E\cup N \cap Q_1| \lesssim \frac{\operatorname{osc}(f,Q_3)}{t}\,.\]
It remains to show that for $t>2\operatorname{osc}(f,Q_3)$
\[\{\overline{\Theta}_f>4nt\}\cap Q_1 \subset E_t\,.\]
Since $f$ is differentiable in $x_0$ for any given $x_0 \in E_t$ we may consider 
\[\tilde{f}(x)=f(x_0+x)-f(x_0)-Df(x_0)x\,.\]
Again, since $x_0 \in E_t$ we can apply \eqref{e.taylor for convex} to $s\mapsto \tilde{f}(se_i)$ for a fixed direction $e_i$to deduce that 
\[ 0\le \tilde{f}(he_i)\le 2t \,h^2 \quad \forall |h|<1\,.\]
Furthermore, for any $x \in K_h=\operatorname{conv}\{ \pm h e_i\colon i=1\dotsc n\}$ we deduce appropriate $0\le \lambda_i,\mu_i\le 1$ that $0\le \tilde{f}(x)\le \sum_i \left(\lambda_i f(he_i)+\mu_i f(-he_i)\right)\le 2t\, h^2$. But since $B_{\nicefrac{h}{\sqrt{n}}} \subset K_h$ we deduce the claim.

\section*{Acknowledgment}
I am deeply grateful to Riccardo Tione and Daniel Faraco for listening patiently to my unstructured thoughts. I would also like to thank Zhuolin Li for bringing the problem to my attention and Pawe\l{} Goldstein for making me aware that this is not the first approach to use viscosity techniques. Finally, I'd like to thank Bernd Kirchheim for hinting me to Maly's amazing idea, encouraging me to include the appendix, and thereby substantially improving the first draft.
\bibliographystyle{plain}
\bibliography{lit}
\end{document}